\documentclass[a4paper,10pt]{amsart}
\usepackage[utf8]{inputenc}

\usepackage{amssymb}
\usepackage{amsthm}
\usepackage{amsmath,amscd}
\usepackage[mathscr]{euscript}
\usepackage[all]{xy}
\usepackage{lmodern}
\usepackage[T1]{fontenc}
\usepackage[textwidth=14cm,hcentering]{geometry}
\usepackage[colorlinks=true,linkcolor=red,citecolor=blue]{hyperref}
\usepackage{cleveref}
\usepackage{mathtools}
\usepackage{xspace}
\usepackage{tikz}
\usetikzlibrary{cd}
\usetikzlibrary{arrows,positioning}

\newtheorem{thm}{Theorem}[section]
\newtheorem{lm}[thm]{Lemma}
\newtheorem{prop}[thm]{Proposition}
\newtheorem{cor}[thm]{Corollary}

\newtheorem*{thm*}{Theorem}

\theoremstyle{definition}
\newtheorem{defn}[thm]{Definition}
\newtheorem{cons}[thm]{Construction}

\newtheorem{nota}[thm]{Notation}

\newtheorem{ex}[thm]{Example}

\newtheorem{rmk}[thm]{Remark}
\newtheorem{ques}[thm]{Question}

\newtheorem{obs}[thm]{Observation}
\newtheorem*{rmk*}{Remark}

\newtheorem{thmx}{Theorem}

\newcommand{\tel}{\mathrm{tel}}
\newcommand{\trans}{\mathrm{trans}}
\newcommand{\Proj}{\mathrm{Proj}}
\newcommand{\cat}{\mathbf}

\newcommand{\on}{\operatorname}
\newcommand{\id}{\mathrm{id}}

\newcommand{\map}{\on{Map}}
\newcommand{\Map}{\on{Map}}

\newcommand{\N}{\mathbb N}
\newcommand{\Z}{\mathbb Z}
\newcommand{\Sph}{\mathbb S}
\newcommand{\CMon}{\mathrm{CMon}}

\newcommand{\Ss}{\cat S}

\newcommand{\Fin}{\mathrm{Fin}}

\newcommand{\Mod}{\cat{Mod}}

\newcommand{\End}{\mathrm{End}}

\newcommand{\pt}{\mathrm{pt}}
\newcommand{\gp}{\mathrm{gp}}
\newcommand{\colim}{\mathrm{colim}}

\newcommand{\coCAlg}{\mathrm{coCAlg}}

\newcommand{\Spec}{\mathrm{Spec\,}}
\newcommand{\gr}{\mathrm{gr}}
\newcommand{\F}{\mathrm{Fin}^\simeq}

\DeclareFontFamily{U}{min}{}
\DeclareFontShape{U}{min}{m}{n}{<-> udmj30}{}

\newcommand{\category}{\xspace{$\infty$-category}\xspace}
\newcommand{\categories}{\xspace{$\infty$-categories}\xspace}

\makeatletter
\let\c@equation\c@thm

\makeatother

\title{$p$-perfection and group completion of $\mathbb{E}_\infty$-monoids}

\author{Maxime Ramzi}
\address{FB Mathematik und Informatik \\
Universit\"at M\"unster \\
Germany} 
\email{\href{mailto:mramzi@uni-muenster.de}{mramzi@uni-muenster.de}}
\urladdr{\url{https://sites.google.com/view/maxime-ramzi-en/home}}

\author{Maria Yakerson}
\address{CNRS \& IMJ-PRG\\
Paris\\
France}
\email{\href{mailto:yakerson@imj-prg.fr}{yakerson@imj-prg.fr}}
\urladdr{\url{https://www.muramatik.com}}

\date{\today}

\begin{document}

\begin{abstract}
We study $\mathbb E_\infty$-monoids on which a prime $p$ acts invertibly, which we call $p$-perfect, in the non-group-complete situation. In particular, we prove that in many examples, they almost embed in their group-completion.

We further study the $p$-perfection functor, and describe it in terms of Quillen's $+$-construction, similarly to group-completion. This gives an alternative description of the $p$-inverted higher algebraic $K$-theory of a ring.
\end{abstract}

\maketitle

\parskip 0.2cm

\parskip 0pt
\tableofcontents

\parskip 0.2cm
\vspace{-2em}

\section*{Introduction}
In homotopy theory, many constructions which classically appear rather harmless, such as group completion of commutative monoids, become much more complicated. For example, the sphere spectrum is the (homotopical) group completion of the groupoid of finite sets, and while the latter is an ordinary $1$-groupoid, the former has nontrivial homotopy groups in arbitrarily high positive degrees. 

Thus, it is somewhat of a miracle that on connective spectra\footnote{This is true more generally for spectra, with the same proof, but we will later focus on non-grouplike commutative monoids.}, also known as grouplike $\mathbb E_\infty$-monoids, or grouplike commutative monoids, the process of inverting a prime $p$ or a set of primes $S$ is rather harmless: it is given by a simple telescope construction $$\colim(X\xrightarrow{p}X\xrightarrow{p}X\to \dots).$$
In particular, the homotopy groups of $X[\frac{1}{p}]$ are simply those of $X$ where we formally invert $p$ in the classical sense (see e.g. \cite[Theorem 17]{thomas}).

However, the proof of this fact truly uses the grouplike-ness of spectra, and more specifically, the fact that their underlying spaces are simple. This paper stems from the observation that this miracle fails away from the grouplike situation --- it fails already for the groupoid of finite sets $\F$, as we explain in detail in \Cref{prop:counterex} and \Cref{rmk:counterexfurther}. 

Our goal in this paper is to study (not-necessarily-grouplike) commutative monoids in the $\infty$-category of spaces, aka $\mathbb E_\infty$-monoids, on which a prime $p$ acts invertibly. We call such commutative monoids \textit{$p$-perfect}, and we also study the \textit{$p$-perfection} functor $(-)[\frac{1}{p}]$.  

The first thing we prove is some kind of elementary structure theorem that allows us to understand the failure of the telescope construction a bit better --- the following is \Cref{thm:simple}:
\begin{thmx}
    Let $M$ be a $p$-perfect commutative monoid. Then:
    \begin{enumerate}
        \item $M$ is a simple space, that is, $\pi_1(M,m)$ is abelian for every $m \in M$, and its action on $\pi_n(M,m)$ is trivial for all $n\geq 1$; 
        \item $\pi_n(M,m)$ is uniquely $p$-divisible for all $n\geq 1$. 
    \end{enumerate}
\end{thmx}

With this in hand, the similarity with group-complete commutative monoids becomes somewhat striking, and in fact we are able to compare $p$-perfection of a commutative monoid and the $p$-perfection of its group completion in an \textit{a priori} rather surprising way. We refer to \Cref{def:locmon} for the definition of ``locally monogenic'' --- we simply point out here that for $R$ a commutative ring with connected $\Spec(R)$, the groupoid of finitely generated projective $R$-modules $\Proj(R)^\simeq$ is locally monogenic, and so is $\F$. With this definition, the following is \Cref{thm:group completion}:
\begin{thmx}
    Let $M$ be a locally monogenic commutative monoid. Consider the pullback square of commutative monoids, induced by the group completion map $M \to M^\gp$:
    
   \[\begin{tikzcd}
	N & M^\gp[\frac{1}{p}]  \\
\pi_0(M)[\frac{1}{p}] & \pi_0(M)^\gp[\frac{1}{p}] 
	\arrow[from=1-1, to=1-2]
	\arrow[from=1-1, to=2-1]
	\arrow[from=1-2, to=2-2]
	\arrow[from=2-1, to=2-2]
\end{tikzcd}\]

Then the natural map $M[\frac{1}{p}] \to N$ is an isomorphism on $\pi_0$ and an equivalence at all components except 0.
\end{thmx}
Finally, again by analogy with group-completion, rather than try to compute the $p$-perfection functor on the nose (which, by the above, may be as hard as group completion!), we give a way to describe its homology, analogously to the group-completion theorem. In the following statement, the $+$-superscript refers to Quillen's $+$-construction (see \cite{hoyois+} for a modern account), and $\tel_p(M)$ refers to the telescope construction $\colim(M\xrightarrow{p}M\xrightarrow{p}M\to \dots)$. With this notation, the following is \Cref{thm:plus}: 
\begin{thmx}
      Let $M$ be a commutative monoid. There are canonical maps of commutative monoids $$\tel_p(M)^+\to \tel_p(M^+)\to M[\frac{1}{p}]$$ that are both equivalences. 
\end{thmx}
In particular, this theorem provides a homology isomorphism between $\tel_p(M)$ and $M[\frac{1}{p}]$. In \Cref{thm:homology}, we give a second proof of this homology isomorphism with $\mathbb F_p$-coefficients using the ``bialgebra Frobenius'' on $H_*(M;\mathbb F_p)$. 

This theorem also suggests a comparison with a slightly different formal inversion of $p$,  namely the one where we view $M$ as a module (in spaces) over the groupoid of finite sets equipped with the \emph{cartesian} product and invert $p$ as a module map. One can also imagine remembering only the module structure over the full subgroupoid spanned by $p$. Informally speaking, this means that we only remember how to add ``elements'' in $M$ to themselves a number of times, or a $p$-power number of times respectively, and invert $p$ in that context. 
Surprisingly, all these three inversions give the same result, as we prove in \Cref{cor:agreemarc}. 
\begin{rmk*}
    We note that while we use the symbol $p$, $p$ does not need to be a prime, it simply needs to be a natural number $\geq 2$. Not only are the proofs the same, but the relevant results also follow from the case of a prime. 
\end{rmk*}
\subsection*{Linear overview}
The early sections match the order in which we stated our main results: \Cref{sec:pperf} deals with generalities about $p$-perfect commutative monoids and there we prove \Cref{thm:simple}; \Cref{sec:group} is about the proof of \Cref{thm:group completion} and \Cref{section:+-construction} deals with the proof of \Cref{thm:plus}. \Cref{section:module} is concerned with the comparison of $p$-perfection with another formal inversion of $p$, and there we prove \Cref{cor:agreemarc}. We added an extra section, \Cref{section: bialgebra}, which gives a different proof of a special case of \Cref{thm:plus}, because we thought the proof technique worth recording\footnote{It also happens that this proof was obtained earlier than \Cref{thm:plus}.}.

\subsection*{Related work}
This work arose partly from a discussion with Ben Antieau about his work in~\cite[Section 2.2]{ben}. In the meantime, he obtained a different proof of \Cref{thm:homology} (which also implies \Cref{cor:homologypperf}), cf. \cite[Corollary 2.10]{ben}.
\subsection*{Notation}
We freely use the language of \categories as developed by Lurie in \cite{HTT,HA}. 

$\Ss$ denotes the \category of spaces. $\CMon$ denotes the \category of commutative monoids in $\Ss$, also known as $\mathbb E_\infty$-monoids. Every monoid $M \in \CMon$ has canonical Frobenius map $p\colon M\to M$.
$\CMon[\frac{1}{p}]$ denotes the full subcategory of monoids $M$ such that Frobenius map $p\colon M\to M$ is an equivalence. The inclusion $\CMon[\frac{1}{p}] \hookrightarrow \CMon$ admits a left adjoint, which we denote $M\mapsto M[\frac{1}{p}]$. We denote mapping spaces by $\Map$.

$\F$ denotes the groupoid of finite sets, which we will consider as a commutative monoid under disjoint union. 
$\Proj(R)^\simeq$ denotes the groupoid of finitely generated projective $R$-modules, considered as a commutative monoid under direct sum. 

\subsection*{Acknowledgements}
We are indebted to Ben Antieau for helpful discussions, and for his suggestion to investigate \Cref{thm:group completion} and \Cref{thm:plus}; as well as to Marc Hoyois for a discussion that led to \Cref{lm:prod} and eventually to \Cref{cor:agreemarc}.

We are also grateful to Andrea Bianchi for helpful comments, and a discussion of the group $A_{p^3}$.

Maxime Ramzi is funded by the Deutsche Forschungsgemeinschaft (DFG, German Research Foundation) -- Project-ID 427320536 -- SFB 1442, as well as by Germany's Excellence Strategy EXC 2044 390685587, Mathematics Münster: Dynamics--Geometry--Structure, and in the beginning stages of the research presented here was supported by the Danish National Research Foundation through the Copenhagen Centre for Geometry and Topology (DNRF151). 

Maria Yakerson is grateful to CNRS and Institut de Mathématiques de Jussieu --- Paris Rive Gauche for perfect working conditions which allowed to pursue this project.

Finally, we are grateful to the Max Planck Institute for Mathematics in Bonn for its hospitality during the workshop on Dualisable Categories \& Continuous $K$-theory, during which we were able to discuss the contents of this article with Ben Antieau. 
\section{Homotopy groups of $p$-perfect monoids}\label{sec:pperf}

\begin{defn}
Let $M$ be a commutative monoid. $M$ is called \textit{$p$-perfect} if it belongs to $\CMon[\frac{1}{p}]$, i.e., if the Frobenius map $p:M\to M$ is an equivalence.
\end{defn}

In this section, we analyze the structure of homotopy groups of $p$-perfect commutative monoids. The main result of this section is the following theorem.

\begin{thm}\label{thm:simple}
    Let $M$ be a $p$-perfect commutative monoid. Then:
    \begin{enumerate}
        \item $M$ is a simple space, that is, $\pi_1(M,m)$ is abelian for every $m \in M$, and its action on $\pi_n(M,m)$ is trivial for all $n\geq 1$; 
        \item $\pi_n(M,m)$ is uniquely $p$-divisible for all $n\geq 1$. 
    \end{enumerate}
\end{thm}

Before proving \Cref{thm:simple}, we would like to point out two differences with the grouplike case that are at the heart of our difficulties. First, not all components of a commutative monoid $M$ are the same, and in particular we may not reduce to the component of $0\in~M$. Specifically, the Frobenius map typically goes from $(M,m)$ to $(M, pm)$ and is thus not an endomorphism of pointed spaces, which complicates $\pi_*$-considerations. Second, for essentially the same reason, $M$ itself need not be a simple space: its fundamental groups can be arbitrarily noncommutative (to wit, $\F$ has $\Sigma_n$ as the fundamental group at the basepoint $n$), and they can act non-trivially on higher homotopy groups.

The idea to prove \Cref{thm:simple} is to consider the following map of pointed spaces: \[(M,m)^p\xrightarrow{\mu} (M,pm) \xrightarrow{\simeq} (M,m)\] where the first map is the addition, and the second map is the inverse of the Frobenius map (which is an equivalence by the assumption that $M$ is $p$-perfect). 

    This induces, for $n \geqslant 1$, a family of maps
   \begin{equation}\label{eq: rho_n}
 \rho_n:~\pi_n(M,m)^p\to~\pi_n(M,m)
\end{equation}
  satisfying certain conditions coming from properties of $M$, thus providing an extra algebraic structure on these homotopy groups. The general claim is that this algebraic structure is enough to guarantee the conclusions of \Cref{thm:simple}.

    First, we will analyse the extra structure on the fundamental groups of $M$. 

    \begin{prop}\label{prop: rho}
        Let $G$ be a group with a group morphism $\rho: G^p\to G$ such that: 
        \begin{enumerate}
            \item For all $g\in G$, $\rho(g,...,g) = g$; 
            \item For any permutation $\sigma\in\Sigma_p$, there is a fixed $k_\sigma\in G$ such that $\rho\circ \sigma = k_\sigma \rho k_\sigma^{-1}.$
        \end{enumerate}

        In this case, $G$ is abelian and uniquely $p$-divisible. Furthermore, the map $\rho_1$ defined in (\ref{eq: rho_n}) satisfies these two conditions with $G=\pi_1(M,m)$. 
    \end{prop}
    \begin{proof}
        Let us first justify that $\rho_1:\pi_1(M,m)^p\to \pi_1(M,m)$ satisfies these two conditions. The first one is evident, as the composite $(M,m)\xrightarrow{\Delta}(M,m)^p\xrightarrow{\mu} (M,pm)\xrightarrow{\simeq}(M,m)$ is by definition the composition of an equivalence and its inverse, i.e., the identity. 

        For the second one, this comes from the fact that the addition  on $M$ is homotopy commutative. Note that the homotopy commutativity of the addition does not give pointed homotopies, it only gives homotopies, hence the need for conjugation by $k_\sigma$. In fact, for a general $M$, $k_\sigma$ can be viewed as a power operation applied to $m$.

        Now let us prove the algebraic claim. Specializing the second condition to diagonal elements, and using the first condition, we find that for any $g$ and $\sigma$, $$g=\rho(g,...,g)= \rho\circ \sigma(g,...,g) = k_\sigma \rho(g,...,g) k_\sigma^{-1} = k_\sigma g k_\sigma^{-1}.$$ Hence $k_\sigma$ is central, and so the second condition simplifies to $\rho\circ \sigma = \rho$. 

    In particular, we find: $$g=\rho(g,...,g) = \rho(g,1,...,1)\rho(1,g,1,...,1)... \rho(1,...,1,g) = \rho(g,1,...,1)^p$$ so that $G$ is indeed $p$-divisible. Moreover since $ g \mapsto \rho(g,1,...,1)$ is a group homomorphism, $G$ is uniquely $p$-divisible. We also find $gh= \rho(g,1,...,1)^p \rho(1,h,1,...)^p$, and since $(g,1,...,1)$ and $(1,h,1,...)$ commute in $G^p$, we get that $gh= hg$ and we deduce that $G$ is abelian.

    \end{proof}
   
    \begin{rmk}
    Note that under the conditions of \Cref{prop: rho} we get that $\rho: G^p\to G$ is simply equal to $\frac{1}{p}\sum_{i=1}^p$. 
    \end{rmk}
    
    We can then argue similarly for higher homotopy groups. The maps \[\rho_n:~\pi_n(M,m)^p \to~\pi_n(M,m),\] defined in (\ref{eq: rho_n}), have the same first property as in \Cref{prop: rho}, and in the second property the conjugation by $k_\sigma$ is replaced by the \emph{action} of $k_\sigma\in\pi_1(M,m)$ on $\pi_n(M,m)$ which we denote as $\rho_n\circ \sigma = k_\sigma\cdot \rho_n$. To sum up, we want to prove the following statement.
    
    \begin{prop}\label{prop: rho_n}
        Let $G$ be a group acting on an abelian group $A$. Assume there is a morphism $\rho: A^p\to A$ of abelian groups such that: 
         \begin{enumerate}
            \item For all $a\in A$, $\rho(a,...,a) = a$; 
            \item For any permutation $\sigma\in\Sigma_p$, there is a fixed $k_\sigma\in G$ such that $\rho\circ \sigma = k_\sigma \cdot\rho $.
        \end{enumerate}
        In this case, $A$ is uniquely $p$-divisible. If, furthermore, $G$ is abelian, uniquely $p$-divisible, and the map $\rho: A^p \to A$ is compatible with the action of $G^p$ (where $G^p$ acts on $A$ via $G^p \xrightarrow{\frac{1}{p}\sum_{i=1}^p}G$), then the action of $G$ on $A$ is trivial.

        Furthermore, each map $\rho_n$ defined in (\ref{eq: rho_n})  provides the group $A=\pi_n(M,m)$ with this structure, where  $G=\pi_1(M,m)$ with the standard action of $G$ on $A$.
    \end{prop}
    \begin{proof}
       The proof that homotopy groups of $M$ have this extra structure with these properties is exactly the same as in the proof of \Cref{prop: rho}. 

       For the algebraic fact, we may again plug in a diagonal element, to find that $$a= \rho(a,...,a) = \rho\circ \sigma(a,...,a) = k_\sigma \cdot \rho (a,...,a) = k_\sigma \cdot a$$ from which it follows that $k_\sigma$ acts trivially on $\pi_n(M,m)$. So we have the same argument for unique $p$-divisibility as in \Cref{prop: rho}. 

Now for the trivial action, we note that $\rho: A^p \to A$ must also be $\frac{1}{p}\Sigma_{i=1}^p$ by the previous computation. So the equivariance condition for $\rho$ reads as follows: for $\underline g\in G^p, \underline a\in A^p$, $$\frac{1}{p}\sum_{i=1}^p( g_i\cdot a_i)= (\frac{1}{p}\sum_i g_i)\cdot (\frac{1}{p}\sum_i a_i).$$
Putting all of the $g_i$'s except $g_1$ to be trivial, and all of the $a_i$'s except for $a_p$ to be trivial, we find $\frac{1}{p}a_p= (\frac{1}{p}g_1)(\frac{1}{p}a_p)$. Multiplying this by $p$, we find that $\frac{1}{p}g_1$ acts trivially on $a_p$. However both $g_1$ and $a_p$ were chosen to be arbitrary, so the action is indeed trivial. 
    \end{proof}
\begin{proof}[Proof of \Cref{thm:simple}]
   For a commutative monoid $M$, combine \Cref{prop: rho} and \Cref{prop: rho_n} to find that all the actions of $\pi_1(M,m)$ on $\pi_n(M,m)$ are trivial for all basepoints $m\in M$, and all $\pi_n(M,m)$ are uniquely $p$-divisible. 
\end{proof}
\begin{cor}\label{cor:homologypperf}
    Let $M$ be a $p$-perfect commutative monoid. The map $M\to\pi_0(M)$ is an $\mathbb F_p$-homology isomorphism. 
\end{cor}

\begin{proof}
    Indeed, from \Cref{thm:simple} and (a slight extension of) Serre's version of Hurewicz theorem (see e.g. the main theorem of \cite{modC}, p.55)  it follows that for a $p$-perfect monoid $M$, the groups $H_*(M;\mathbb F_p)$ are $p$-divisible for ${* \geqslant 1}$. Given that these groups are $\mathbb F_p$-vector spaces, it follows that they all vanish except $H_0(M;\mathbb F_p) \simeq H_0(\pi_0 (M);\mathbb F_p)$.
\end{proof}

\begin{rmk}
    It is not hard to prove that $\pi_0(M[\frac{1}{p}]) \cong \pi_0(M)[\frac{1}{p}]$ where the latter is meant in the classical sense (in particular, it equals $\pi_0 (\tel_p(M))$), so we can completely compute $H_*(M[\frac{1}{p}];\mathbb F_p)$. We will improve on this calculation in the rest of the paper.  
\end{rmk}

Next, we consider the following ``approximation'' to $p$-perfection.


\begin{cons}\label{cons:telp}
    Let $M$ be a commutative monoid. The \textit{$p$-telescope} of M is  \[\tel_p(M) := \colim(M\xrightarrow{p}M\xrightarrow{p}\dots),\] the colimit being taken in $\CMon$. 
\end{cons}

\begin{ques}\label{ques: telescope}
For a commutative monoid $M$, is the canonical map $\tel_p(M) \to M[\frac{1}{p}]$ an equivalence?
\end{ques}

For discrete commutative monoids, as well as grouplike commutative monoids (and more generally spectra) the answer to \Cref{ques: telescope} is indeed positive. However, using \Cref{thm:simple} we can see that in general it can be negative, as in the following example.

\begin{prop}\label{prop:counterex}
    Let $M=\F$. In this case, the canonical map $\tel_p(M)\to M[\frac{1}{p}]$ is \emph{not} an equivalence, i.e., the monoid $\tel_p(M)$ is \emph{not} $p$-perfect. 
\end{prop}

\begin{proof}
Fix $1\in \F$ and let $\underline 1 \in \tel_p(\F)$ be its image under the canonical map. Since the fundamental group functor commutes with filtered colimits, we get: $$\pi_1(\tel_p(\F),\underline 1)\cong \colim_k \pi_1(\F, p^k) \cong \colim_k \Sigma_{p^k}$$ and each of the maps $\Sigma_{p^k}\to \Sigma_{p^{k+1}}$ sends $\sigma$ to $(\sigma, ... ,\sigma)$ and in particular is injective. 

Thus $\Sigma_{p}$ embeds in $\pi_1(\tel_p(\F),\underline 1)$, which is therefore not abelian. It follows from \Cref{thm:simple} that $\tel_p(\F)$ is not $p$-perfect, and thus not equivalent to $\F[\frac{1}{p}]$. 
\end{proof}
\begin{rmk}\label{rmk:counterexfurther}
    With a similar argument, one may prove that the image of the group $\Sigma_p$ in $\pi_1(\tel_p(\F),\underline 1)$ is not $p$-divisible. The moral here is that $\sigma \in \Sigma_p$ is indeed sent to $(\sigma, ..., \sigma) = (\sigma, 1,...,1)(1,\sigma,1,...,1)\dots(1,...,1, \sigma) \in \Sigma_{p^2}$ \emph{but} each of the 
    $(1,...,1,\sigma,1,...)$ differs from the others by a conjugation. We explain this conjugation below. 

The failure of $\tel_p(\F)$ to be $p$-perfect is closely related to the failure of telescopes to be group complete, cf. \cite{thomas}. The idea here is that the Frobenius map on the space $\colim (M\to M\to ...)$ is induced by a map of diagrams, namely a map which looks like: 
\[\begin{tikzcd}
	M & M & M & \dots \\
	M & M & M & \dots
	\arrow["p", from=1-1, to=1-2]
	\arrow["p"', from=1-1, to=2-1]
	\arrow["p", from=1-2, to=1-3]
	\arrow[""{name=0, anchor=center, inner sep=0}, "p", from=1-2, to=2-2]
	\arrow[from=1-3, to=1-4]
	\arrow["p", from=1-3, to=2-3]
	\arrow[from=1-4, to=2-4]
	\arrow["p"', from=2-1, to=2-2]
	\arrow["p"', from=2-2, to=2-3]
	\arrow[from=2-3, to=2-4]
	\arrow["h"', draw=none, from=1-1, to=0]
\end{tikzcd}\]
The homotopy $h$ that is here is an \emph{interesting} homotopy, it is the one that witnesses that $p:M\to M$ is a map of commutative monoids, and hence commutes with the ``scalar'' $p$. It is a \emph{different} one than the homotopy $f\circ f= f\circ f$ that exists for any endomorphism $f$. If it were the latter, then the induced map on colimits would in fact be an equivalence, but because it is the former, there is no reason for it to be the case. 
This difference in homotopies is related to the fact that $(\sigma, ..., \sigma)$ is $p$-divisible \emph{up to conjugacy} in $\Sigma_{p^2}$ --- note that conjugacies in $\pi_1$ arise precisely from \emph{unpointed} homotopies, and they influence the induced map on colimits. 
\end{rmk}
In fact, we can make that precise with the following result --- this is analogous to \cite[Proposition 6]{thomas}, and the method is similar so for simplicity we only state the parts that will be relevant to the rest of the paper:
\begin{prop}\label{prop: sigma^3}
    Let $M$ be a commutative monoid. 
    \begin{enumerate}
        \item \label{item: tel p-perfect} Suppose $\tel_p(M)$ is $p$-perfect. Then it is the $p$-perfection of $M$. 
        \item Suppose $\pi_1(M,x)$ is hypoabelian for all $x$. Then $\tel_p(M)$ is $p$-perfect. More generally, $\tel_p(M)$ is $p$-perfect if for every $m\in M$, the image of $p^{\textnormal{3-cycle}}\in \Sigma_{p^3}$ in $\pi_1(M,p^3 m)$ is trivial, or the same statement with $3$ replaced by some $n\geq 3$.  
        \item More generally, suppose that for some $n\geq 3$ the image of $p^{n\textnormal{-cycle}}$ in $\pi_1(\tel_p(M), m)$ is $0$ for every $m\in M$. Then $\tel_p(M)$ is $p$-perfect. 
    \end{enumerate}
\end{prop}
\begin{proof}
The proofs are the same as in \cite[Proposition 6]{thomas} (see also \Cref{cor:agreemarc} for a formal relation between the two statements --- note that the latter does not depend on this statement, so there is no circularity).

    Let $N$ be a commutative monoid. Since $p: X\to X$ is a natural transformation of the identity on commutative monoids, the following two maps $$\map_{\CMon}(p\cdot -, N): \Map_{\CMon}(M,N)\to \Map_{\CMon}(M,N)$$ and $$\map_{\CMon}(M, p\cdot -):\Map_{\CMon}(M,N)\to \Map_{\CMon}(M,N)$$ are homotopic. Therefore, if $N$ is $p$-perfect, the first one is an equivalence because the second one is. It follows that in this case, restriction along $M\to \tel_p(M)$ induces an equivalence $$\map_{\CMon}(\tel_p(M),N)\xrightarrow{\simeq}\map_{\CMon}(M,N).$$
    The first point then follows immediately by the universal property of $M[\frac{1}{p}]$. 

    For the second point, recall that a group is hypoabelian if it contains no nontrivial perfect subgroups. Since the $p^{3-\textnormal{cycle}}$ is a product of disjoint $3$-cycles, it belongs to $A_{p^3}\subset \Sigma_{p^3}$, which is perfect for all $p$ The images of perfect subgroups are perfect, hence the hypoabelian case is indeed a special case of the ``More generally'' statement. 

    Suppose then that the $p^{3-\textnormal{cycle}}$ is trivial in each $\pi_1(M,p^3m)$, and consider the diagram of commutative monoids $M\xrightarrow{p^2}M\xrightarrow{p^2 }M \to \dots $. 

    By \emph{naturality} of multiplication by $p$, this induces a commutative diagram of spaces --- we added  basepoints but it is not a priori a diagram of pointed spaces:
    \[\begin{tikzcd}
	{(M,m)} & {(M,p^2m)} & {(M,p^4 m)} & \dots & {(M,p^{2k}m)} \\
	{(M,pm)} & {(M,p^3m)} & {(M,p^5m)} & \dots & {(M,p^{2k+1}m)}
	\arrow["{p^2}", from=1-1, to=1-2]
	\arrow["p"', from=1-1, to=2-1]
	\arrow["{p^2}", from=1-2, to=1-3]
	\arrow["p"', from=1-2, to=2-2]
	\arrow["p"', from=1-3, to=2-3]
	\arrow["p"', from=1-5, to=2-5]
	\arrow["{p^2}"', from=2-1, to=2-2]
	\arrow["{p^2}"', from=2-2, to=2-3]
\end{tikzcd}\]

Here the $2$-cells are obtained using the fact that $p^2$ is a map of commutative monoids and $p$ is natural with respect to those. 

Unwinding this, we find that the homotopy $p\circ p^2\simeq p^2\circ p$ sends the point $m$ to the image of $p^{3-\textnormal{cycle}}\in \Sigma_{p^3}$ in $\pi_1(M,p^3m)$. 

This element is trivial in each $\pi_1(M,p^{2k+3}m)$ by assumption, and so the $2$-cells are homotopies of pointed maps, and so this map of diagrams lifts to a map of diagrams of pointed spaces. Then we may take $\pi_n$, and we find that the vertical map induces an isomorphism on $\pi_n$ on colimits. Since $\pi_n$ commutes with colimits, this implies that $p: \tel_p(M)\to \tel_p(M)$ induces an isomorphism on $\pi_n$ at every point and for all $n$, and is thus an equivalence, which is what we wanted to prove. 

More generally, if the image of $p^{3\textnormal{-cycle}}$ only vanishes in $\pi_1(\tel_p(M),m)$, we know it must vanish at some finite stage of the telescope, and we can run the same argument simply by cofinality. For details about this last point, see \cite{thomas}.
\end{proof}

\begin{rmk}
In particular, if $\pi_1(M)$ is hypoabelian at all points, then $\tel_p(M)$ is simple by \Cref{thm:simple} and hence hypoabelian as well. This is not \emph{a priori} clear, in the sense that in general, a filtered colimit of hypoabelian groups need not be hypoabelian. To give a concrete example, note that subgroups of free groups are free, so that if $F$ is a free group on countably many generators, there is an abstract isomorphism $f: F\cong [F,F]$. In particular, we may follow it up with the inclusion $[F,F]\subset F$ to build a sequence $F\to F\to F \to \dots$ of inclusions whose filtered colimit is a perfect group, while each $F$ is hypoabelian. 
\end{rmk}

As in \cite[Example 5.2]{algcob}, the above also works in the context of motivic spaces, i.e., $\mathbb A^1$-invariant sheaves on the site of smooth $S$-schemes $\mathrm{Sm}_S$ for some scheme $S$. Considering, for example, $M=\mathrm{Vect}\simeq\coprod_n \mathrm{BGL}_n$, we find that for any ring $R$, the relevant matrix corresponding to $p^{\textnormal{3-cycle}}$ is given by  permutation matrix. 

Since this matrix belongs to $\mathrm{SL}_{p^3}(\mathbb Z)$ (determinants of permutation matrices agree with the signature of the permutation), it is $\mathbb A^1$-homotopic to the identity. Thus we find:
\begin{cor}
   The canonical map of presheaves of spaces on $\mathrm{Sm}_S$ 
   $$\tel_p (\mathrm{Vect})\to \mathrm{Vect}[\frac{1}{p}]$$
   is an $\mathbb A^1$-homotopy equivalence on affines. In particular, it is a motivic equivalence.
\end{cor}


\section{Comparison of $p$-perfection and group completion}\label{sec:group}

In this section, we prove our main result for many commutative monoids, including $\F$ and $\Proj(R)^\simeq$ with connected $\Spec R$: after $p$-perfection, connected components of such commutative monoid embed as ``positive'' components of its group completion, except for the component of 0. 

First, we introduce some notation.

\begin{nota}
    Let $x \in M$ be a point in a commutative monoid. 
    \begin{itemize}
        
    \item The \textit{telescope of $M$ at $x$} is 
    \[\tel^x(M) := \colim(M\xrightarrow{+x}M\xrightarrow{+x}\dots),\] the colimit being taken in $\Ss$.

    \item The formal inversion of $x$ on $M$ is $M[-x]$,  which is typically not the same as the telescope $\tel^x(M)$. 
    \end{itemize}
\end{nota}

\begin{defn}\label{def:locmon}
    Let $M$ be a commutative monoid. We say that $M$ is \textit{locally monogenic} if for any non-zero $x, y 
    \in M$ there exists $n \in \N$ and $z \in M$ such that $nx = y+z$.
\end{defn}

We introduce locally monogenic monoids because of the following nice property, which follows immediately from the definition. 

\begin{prop}\label{prop: invert anything}
    Let $M$ be a locally monogenic monoid. Then for any non-zero $x \in M$ the map $M[-x] \to M^\gp$ is an equivalence.
\end{prop}

\begin{proof}
    Since for a fixed non-zero $x$ and for any $y$ we have a presentation of the form $nx = y + z$, we get that any $y$ becomes invertible in $M[-x]$, so the two share the same universal property.
\end{proof}

\begin{ex}\
The monoid $\F$ is locally monogenic.
\end{ex}

\begin{ex}\
For a ring $R$ with connected $\Spec R$, the monoid $\Proj(R)^\simeq$ is locally monogenic by~\Cref{prop: Proj connected} below.
\end{ex}

\begin{prop}\label{prop: Proj connected}
    Let $R$ be a ring and $P$ a finitely generated projective $R$-module. Assume that the rank of $P$ is a non-vanishing function on $\Spec R$ (for example, $P$ is non-zero and $\Spec R$ is connected). Then there exists $n \in \N$ such that $P^n$ splits off a trivial summand $R$.
\end{prop}

\begin{proof}
    $P$ is a projective module hence locally free. Since $\Spec R$ is quasicompact, there is a finite cover by distinguished open sets  $\Spec R = \bigcup_{i=1}^n D(f_i)$ such that $P[\frac{1}{f_i}]$ is a free $R[\frac{1}{f_i}]$-module for every $i$.

    For each $i$, pick a surjection $\bar{r}_{i} \colon P[\frac{1}{f_i}] \to R[\frac{1}{f_i}]$.
    Since $P$ is finitely generated projective, up to a power of $f_i$ we can lift $\bar{r}_{i}$ to a map $r_i \colon P \to R$. Consider the map of $R$-modules $r = \bigoplus_{i=1}^n r_i \colon  P^{n} \to R$. The map $r$ induces a map of $R$-modules which is a surjection locally on $\Spec R$, hence it is surjective. Since its target is $R$, it automatically splits.
\end{proof}

\begin{thm}\label{thm:group completion}
    Let $M$ be a locally monogenic commutative monoid. Consider the pullback square of monoids, induced by the group completion map $M \to M^\gp$:
    
   \[\begin{tikzcd}
	N & M^\gp[\frac{1}{p}]  \\
\pi_0(M)[\frac{1}{p}] & \pi_0(M)^\gp[\frac{1}{p}] 
	\arrow[from=1-1, to=1-2]
	\arrow[from=1-1, to=2-1]
	\arrow[from=1-2, to=2-2]
	\arrow[from=2-1, to=2-2]
\end{tikzcd}\]

Then the natural map $M[\frac{1}{p}] \to N$ is an isomorphism on $\pi_0$ and an equivalence at all components except 0.

\end{thm}

The statement of \Cref{thm:group completion} has this ``except 0'' condition for a concrete reason. To describe it, we introduce:
\begin{defn}
Let $M$ be a commutative monoid. We say that \textit{$0$ is isolated in $M$} if the following two conditions hold:
\begin{itemize}
    \item The inclusion $0\to M$ is an inclusion of components, i.e., $\Omega(M,0)\simeq \pt$; 
    \item For $x,y\in M, (x+y\simeq 0)\implies (x\simeq y\simeq 0)$
\end{itemize}
\end{defn}
\begin{ex}
Let $C$ be a category with coproducts such that $x\coprod y \simeq 0$ implies $x = 0$ or $y=0$, e.g. if $C$ is semiadditive, or $C= \Fin$. In this case, $0$ is isolated in $(C^\simeq,\coprod)$. 
\end{ex}
\begin{lm}
    Let $M$ be a commutative monoid. $0$ is isolated in $M$ if and only if $M=N_+$ for some non-unital commutative monoid $N$, where $(-)_+$ is the left adjoint to the forgetful functor from commutative monoids to non-unital commutative monoids. 
\end{lm}
\begin{proof}
``If'': By \cite[Proposition 5.4.4.8]{HA}, the underlying space of $N_+$ is $N\coprod \{0\}$, proving the first condition of $0$ being isolated; the map $N\to N_+$ being a map of nonunital commutative monoids shows that the second condition is also satisfied. 

``Only if'': Suppose $0$ is isolated in $M$. By the second condition in the definition of isolated, the full sub-space of $M$ spanned by non-zero components is closed under products and therefore defines a non-unital commutative monoid $N$ with a non-unital commutative monoid map $N\to M$, which induces a map $N_+\to M$. Again by \cite[Proposition 5.4.4.8]{HA}, and now using the first condition in the definition, we see that $N_+\to M$ is an equivalence.
\end{proof}

\begin{prop}\label{prop: isolated}
    If $M$ is a commutative monoid and 0 is an isolated point, then 0 is also isolated in $M[\frac{1}{p}]$.
\end{prop}

\begin{proof}
    If $M = N_+$ where $N$ is a non-unital commutative monoid then $M[\frac{1}{p}] \simeq N[\frac{1}{p}]_+$, where now we consider $(-)[\frac{1}{p}]$ as functor on non-unital commutative monoids. Indeed, $(-)[\frac{1}{p}]$ and $(-)_+$ commute as they are left adjoints of commuting forgetful functors.
\end{proof}

\begin{rmk}
    \Cref{thm:group completion} is not true for an arbitrary commutative monoid, for example, it does not hold for $ M = \Proj(R)^\simeq$ when $R = \Z \times \Z$. Indeed, consider the point $(\Z, 0)$ in $M[\frac{1}{p}] \simeq \Proj(\Z)^\simeq [\frac{1}{p}] \times \Proj(\Z)^\simeq [\frac{1}{p}]$. By \Cref{prop: isolated}, $0$ is isolated in the second summand and so the component of $(\Z,0)$ in this product is equivalent to a component of $K(\Z)[\frac{1}{p}]\times \{0\}$. However, the point $0$ is not isolated in the second summand of $M^\gp[\frac{1}{p}] \simeq K(\Z)[\frac{1}{p}] \times K(\Z)[\frac{1}{p}]$ and the component of $(\Z,0)$ therein is equivalent to a product of a component of $ K(\Z)[\frac{1}{p}]$ with itself. 

    On the other hand, \Cref{thm:group completion} can also be true for a monoid that is not locally monogenic: for example, it holds for formal reasons for the monoid $M = K(\Z)_{\geqslant 0} \times K(\Z)_{\geqslant 0}$, where $K(\Z)_{\geqslant 0} = K(\Z) \times_\Z \N$.
\end{rmk}

We first prove the following partial case of \Cref{thm:group completion}.

\begin{prop}\label{prop: one generator}
Let $M$ be a $p$-perfect commutative monoid such that for some $x \in M$ the group completion map $M[-x] \to M^\gp$ is an equivalence. Denote $T = \tel^x(M)$. Then
the map \[(M, x) \to (T,x) \to (M^\gp, x)\]
is an equivalence between the component of $x$ in $M$ and the component of (the image of) $x$ in $M^\gp$. 
\end{prop}

\begin{proof}
Since $M$ is a simple space by \Cref{thm:simple}, we have
$M[-x] \simeq M^\gp \simeq T$ by~\cite[Corollary 7]{thomas}. 
 We want to show that the map $(M, x) \to (T,x)$ is an equivalence of corresponding connected components of $M$ and $T$. 

Since $M$ is $p$-perfect, the composition $(M,x)\xrightarrow{\Delta}(M,x)^p\xrightarrow{\mu} (M,px)$ is an equivalence. On $G = \pi_n(M, x)$ this composition induces 
$G \to G \times \ldots \times G \to H$ which was computed in \cite[Proposition 2.9]{ben}: this composition sends every element $g$ to  $p \cdot \trans(g)$ where $\trans \colon 
(M, x) \to (M, px)$ is the translation by $(p-1)x$. Since the composition is an isomorphism of groups and $G$ is uniquely $p$-divisible for $n \geqslant 1$ by \Cref{thm:simple}, the map $\trans \colon (M, x) \xrightarrow{+(p-1)x} (M, px)$ must be an equivalence as well. 
Observe that $T = \tel^x(M) \simeq \tel^{(p-1)x}(M)$.
Hence the map $(M, x) \to (T,x)$ is an equivalence of components, as we wanted to show.
\end{proof}

\begin{proof}[Proof of \Cref{thm:group completion}]
The theorem follows immediately from \Cref{prop: one generator} applied to the $p$-perfect locally monogenic monoid $M[\frac{1}{p}]$, thanks to \Cref{prop: invert anything}.
\end{proof}

\begin{ex}
    Let $M=\Fin^\simeq$. In this case, $M^\gp\simeq \Sph$ is the sphere spectrum, so \Cref{thm:group completion} tells us that 
    \[\Fin^\simeq[\frac{1}{p}]\setminus\{0\} \, \simeq \, \Sph[\frac{1}{p}]\times_{\Z[\frac{1}{p}]}(\mathbb N[\frac{1}{p}]\setminus\{0\}), \]
    i.e., the union of positive components of $\Sph[\frac{1}{p}]$. 
\end{ex}



We already know that $p$-perfection of $M$ cannot be computed as $\tel_p(M)$ for general $M$, see~\Cref{prop:counterex}. Now we can deduce a more general statement, similar to the group-completion situation\footnote{But in contrast to the group-complet\emph{e} situation.}.

\begin{cor}\label{cor: not filt colim}
For a general commutative monoid $M$ its $p$-perfection cannot be expressed as a filtered colimit of maps on $M$. 
\end{cor}

\begin{proof}
    If it could, then for $1$-truncated $M$, $M[\frac{1}{p}]$ would also be 1-truncated. However we know that $M^\gp[\frac{1}{p}]$ need not be $n$-truncated for any $n$: $M=\F$ serves as an example.
\end{proof}

We will see how to compute the $p$-perfection functor in the next section.

\section{$p$-perfection via $+$-construction}\label{section:+-construction}

The goal of this section is to give a $+$-construction model for $p$-perfection, analogously to the McDuff-Segal group completion theorem which describes group completion in terms of $+$-constructions. More precisely, we prove: 
\begin{thm}\label{thm:plus}
    Let $M$ be a commutative monoid. There are canonical maps of commutative monoids $$\tel_p(M)^+\to \tel_p(M^+)\to M[\frac{1}{p}]$$ that are both equivalences. 
\end{thm}

Our reference for the $+$-construction is Hoyois' note \cite{hoyois+}. As a consequence of his description of the $+$-construction, we find that it defines a functor $(-)^+:\Ss\to \Ss$ which preserves products, and thus also induces a functor $(-)^+:\CMon\to \CMon$.

\begin{lm}
    Let $M\in\CMon$ be a commutative monoid. The map $M\to M^+$ induces an equivalence $M[\frac{1}{p}]\to M^+[\frac{1}{p}]$.

    In particular, the canonical map $\tel_p(M^+)\to M[\frac{1}{p}]$ obtained from $M[\frac{1}{p}]$ being simple and hence acyclic-local is an equivalence. 
\end{lm}
\begin{proof}
    Observe that $+$-construction commutes with $p$-perfection as left adjoints of commuting functors. The claim then follows directly from the universal properties using that $p$-perfect commutative monoids are simple and hence acyclic-local. 

    The ``In particular'' part follows from the first part and from \Cref{prop: sigma^3}. 
\end{proof}
We can now prove the theorem:
\begin{proof}[Proof of \Cref{thm:plus}]
    The previous lemma deals with the second map. In particular, $\tel_p(M^+)$ is hypoabelian and so the map $M\to M^+$ induces a map $\tel_p(M)^+\to \tel_p(M^+)$. 

    Since $M\to M^+$ is a homology isomorphism, and homology commutes with filtered colimits, we find that this map is a homology isomorphism, and hence (both spaces are acyclic-local) an equivalence, as was to be proved. 
\end{proof}


As a straightforward corollary of \Cref{thm:plus}, we get an alternative approach to computing the $1$-component of the $p$-inverted K-theory space $\Omega_1^\infty (K(R) [\frac{1}{p}])$.

 \begin{cor}\label{cor: K-theory}
     Let $R$ be a ring. Then there is a canonical isomorphism
     \[ \colim_k H_*(\mathrm{BGL}_{p^k}(R); \Z) \xrightarrow{\simeq} H_*(\Omega_1^\infty (K(R) [\frac{1}{p}]); \Z)\]
     where colimit is induced by the map 
     \begin{gather*}
         \mathrm{GL}_{p^k}(R) \to \mathrm{GL}_{p^{k+1}}(R) \\
         A \mapsto \begin{pmatrix}
A & 0 & 0  \\
0 & \dots & 0  \\
0 & 0 & A
     \end{pmatrix}
     \end{gather*}
     
 \end{cor}


 \begin{rmk}
\Cref{cor: K-theory} could be useful, for example, if the direct system \[\{\dots \to \mathrm{GL}_{p^k}(R) \to \mathrm{GL}_{p^{k+1}}(R) \to \dots\}\] is simpler for computations than the standard direct system \[\{\dots \to\mathrm{GL}_n(R) \to \mathrm{GL}_{n+1}(R)\to \dots\}.\] In the case when $R = \mathbb F_q$, Quillen computed all homology groups $H_*(\mathrm{BGL}_{n}(R); \Z[\frac{1}{p}])$ unstably \cite[Theorem 3]{quillen}, so taking a different direct system does not simplify the computation of $K(\mathbb F_q)$. However, for other rings $R$ it could possibly be useful for computing $p$-inverted K-theory of $R$.
 \end{rmk}

 \section{Comparison with an alternative construction}\label{section:module}

In this section we show that the following three informally described ways of inverting $p$ on a commutative monoid agree: inverting $p$ among commutative monoids, inverting it among spaces where one can add points to themselves any number of times, and inverting it among spaces where we can add points any $p$-power number of times. 

To do that, we observe that a result similar to \Cref{thm:plus} appears in  \cite[Proposition 3.2]{algcob}. We describe here the relationship between the two and how we could in principle deduce the results from the previous section from \textit{loc. cit}. Before doing so, let us recall a special case of \textit{loc. cit.} for the convenience of the reader:
\begin{prop}[{\cite[Proposition 3.2, Remark 3.3]{algcob}}]
    Let $Q$ be an $\mathbb E_2$-monoid in spaces and $q\in Q$. For every $Q$-module $E$, the canonical map $$\tel_q(E)\to E[q^{-1}]$$ is acyclic. 
\end{prop}
Here, $\tel_q(E)$ is the colimit of the diagram $E\xrightarrow{q\cdot-} E\xrightarrow{q\cdot -}E\to \dots$ which can canonically be made a diagram of $Q$-modules since $Q$ is $\mathbb E_2$, and $E[q^{-1}]$ is the initial $Q$-module below $E$ on which $q$ acts invertibly. 

The relevant connection here is as follows: let $Q = \End(\id_{\CMon})\simeq (\Fin^\times)^\simeq$ be the endomorphism monoid of the identity on $\CMon$ (or the full submonoid thereof spanned by powers of $p$). There is a forgetful functor $$U:\CMon\to \Mod_{Q}(\Ss)$$
This forgetful functor is what we described in the introduction informally as ``remembering only how to add elements of a commutative monoid to themselves (a $p$-power number of times)''. 

For $M\in \CMon$, $\tel_p(M)$ as defined in \Cref{cons:telp} agrees with $\tel_{p_Q}(M)$ where $M$ is considered as a $Q$-module and $p_Q\in Q$ is the element corresponding to multiplication by $p$.

In particular, this forgetful functor sends $p$-perfect commutative monoids to modules on which $p_Q\in Q$ acts invertibly. Thus we get a commutative square of categories which induces a Beck-Chevalley square: 
\[\begin{tikzcd}
	\CMon & {\Mod_Q} \\
	{\CMon[\frac{1}{p}]} & {\Mod_{Q[p^{-1}]}}
	\arrow[from=1-1, to=1-2]
	\arrow[from=1-1, to=2-1]
	\arrow[shorten <=5pt, shorten >=5pt, Rightarrow, from=1-2, to=2-1]
	\arrow[from=1-2, to=2-2]
	\arrow[from=2-1, to=2-2]
\end{tikzcd}\]
which is, concretely, for a commutative monoid $E$, the map $E[p_Q^{-1}]\to E[\frac{1}{p}]$ obtained from considering the map $E\to E[\frac{1}{p}]$ as a map of $Q$-modules. Furthermore, $\tel_p(E)$ is the same here as in our earlier considerations and we get a commutative triangle: 
\[\begin{tikzcd}
	& {\tel_p(E)} \\
	{E[p_Q^{-1}]} & {E[\frac{1}{p}]}
	\arrow[from=1-2, to=2-1]
	\arrow[from=1-2, to=2-2]
	\arrow[from=2-1, to=2-2]
\end{tikzcd}\]
As a consequence of our work, the right hand map is acyclic, and \cite[Proposition 3.2]{algcob} states that the left hand map also is, therefore the lower horizontal map also is. 

The following question is then natural: \textit{is the bottom horizontal map an equivalence?} We prove below that the answer is yes, which provides an alternative proof of \Cref{thm:plus} (as well as \Cref{prop: sigma^3}). This also justifies the intuition that we explained in the beginning of the section, by taking $Q= \Fin^\times$ and the full sub-monoid spanned by powers of $p$.

We will need the following lemmas: 
\begin{lm}\label{lm:prod}
    Let $Q$ be an $\mathbb E_2$-monoid and $q\in Q$. The functor $[q^{-1}]: \Mod_Q\to \Mod_{Q[q^{-1}]}$ preserves finite products, and in particular the same is true for the composite functor $$\Mod_Q\to~\Mod_{Q[q^{-1}]}\to~\Mod_Q$$
\end{lm}
\begin{proof}


For the duration of the proof, let $\hom$ denote the cartesian internal hom on $\Mod_Q$. Since $Q$ is the free $Q$-module on a point, we find that the underlying space of $\hom(X,M)$ is $\map_Q(Q\times X,M)$.

Since $Q$ is $\mathbb E_2$, we have a natural transformation $q:\id\to \id$ on $\Mod_Q$ which is a $Q$-linear lift of the natural map $q: \mathrm{forget}\to\mathrm{forget} $. By naturality, we get a homotopy between $q_X^*:\Map_Q(X,M)\to \Map_Q(X,M)$ and $(q_M)_*: \Map_Q(X,M)\to \Map_Q(X,M)$, which shows that $q$ acts invertibly on a $Q$-module $M$ if and only if $M$ is $q_X$-local for every $X\in\Mod_Q$. 

We furthermore have, for $X,Y\in\Mod_Q$, that $q_X\times\id_Y$ commutes with $\id_X\times q_Y$ and that their composite is $q_{X\times Y}$. In particular, if $M$ is $q_{X\times Y}$-local, it is also $q_X\times\id_Y$-local. It follows that if $M$ is $q_X$-local for all $X\in\Mod_Q$, then for any $Y\in\Mod_Q$, $\hom(Y,M)$ is also $q_X$-local for all $X$: indeed $\Map_Q(X,\hom(Y,M))\simeq \map_Q(X\times Y, M)$ and $q_X$ on the left corresponds to $q_X\times\id_Y$ on the right.

It follows that the localization $[q^{-1}]: \Mod_Q\to \Mod_{Q[q^{-1}]}$ is compatible with the symmetric monoidal structure in the sense of \cite[Definition 2.2.1.6]{HA}. There is a unique induced symmetric monoidal structure on $\Mod_{Q[q^{-1}]}$ for which $[q^{-1}]$ is strong symmetric monoidal. To conclude, we are left with proving that it is cartesian, or, equivalently, that the forgetful functor $\Mod_{Q[q^{-1}]}\to \Mod_Q$ is strong symmetric monoidal (as opposed to lax). This follows from the fact that a product of $q$-local $Q$-modules is still $q$-local.

\end{proof}
\begin{rmk}
    For convenience, let us spell out a proof of the above without resorting to symmetric monoidal localizations: for $X,Y\in\Mod_Q$ and $M\in\Mod_{Q[q^{-1}]}$, we have a chain of equivalences, where the second and second to last equivalences follow from the observation that $\hom(Z,M)$ is $q$-local for every $Z$: $$\Map_Q(X\times Y,M)\simeq \Map_Q(X,\hom(Y,M))\simeq \Map_Q(X[q^{-1}],\hom(Y,M))\simeq \Map_Q(X[q^{-1}]\times Y,M)$$
   $$ \simeq \Map_Q(Y, \hom(X[q^{-1}],M))\simeq \Map_Q(Y[q^{-1}],\hom(X[q^{-1}],M))\simeq \Map_Q(X[q^{-1}]\times Y[q^{-1}], M).$$
    To conclude from there, we crucially use the fact that $X[q^{-1}]\times Y[q^{-1}]$ is $q$-local. 
\end{rmk}

It follows that $[q^{-1}]$ induces a functor $\CMon(\Mod_Q)\to \CMon(\Mod_{Q[q^{-1}]})$. To conclude, we will need a further lemma. To state it, note that $Q$ acts on commutative monoids in any category with finite products, that is, for every such $C$ we have a forgetful functor $\CMon(C)\to \Mod_Q(C)$. 
\begin{lm}\label{lm:subtle}
    The forgetful functor $\CMon\to \Mod_Q$ preserves products, so it also lifts to $\CMon(\Mod_Q)$. The latter has a forgetful functor to $\Mod_Q(\Mod_Q)\simeq \Mod_{Q\times Q}$, and this has \emph{two} forgetful functors to $\Mod_Q$. The following composites agree: $$\CMon\to \CMon(\Mod_Q)\to \Mod_{Q\times Q}\rightrightarrows \Mod_Q$$
\end{lm}
\begin{proof}
This follows from the fact that the two forgetful functors $\CMon(\CMon)\rightrightarrows \CMon$ agree, which is proved in \cite{GGN}.
\end{proof}

\begin{cor}\label{cor:agreemarc}
    Let $E$ be a commutative monoid. The map $E[p_Q^{-1}]\to E[\frac{1}{p}]$ described above is an equivalence. 
\end{cor}
\begin{proof}
    Since both $U:\CMon\to \Mod_Q$ and $[q^{-1}]$ preserve products, we find that $E[p_Q^{-1}]$ has a canonical commutative monoid structure, and the unit map $E\to E[p_Q^{-1}]$ a canonical commutative monoid structure as well. 
    
    We claim that $E[p_Q^{-1}]$ is $p$-perfect, from which the statement  will follow. For this, we need to identify $p$ and $p_Q$ on it. We have the following commutative diagram in $\Mod_Q$: 
    \[\begin{tikzcd}
	E & {E[p_Q^{-1}]} \\
	E & {E[p_Q^{-1}]}
	\arrow[from=1-1, to=1-2]
	\arrow["{p}", from=1-1, to=2-1]
	\arrow["{p}", from=1-2, to=2-2]
	\arrow[from=2-1, to=2-2]
\end{tikzcd}\]
By the universal property of $E[p_{Q^-1}]$, to prove that $p$ and $p_Q$ are equivalent on it, it suffices to prove so after restricting to $E$, and since we also have a commutative diagram like the above with $p_Q$, it suffices to actually prove that in $\Mod_Q$, $p: E\to E$ and $p_Q: E\to E$ are homotopic, which is the content of \Cref{lm:subtle}. 
\end{proof}
\section{The bialgebra Frobenius and colimit perfections}\label{section: bialgebra}

Our goal in this section is to provide an alternative proof of the following theorem.

\begin{thm}\label{thm:homology}
Let $M$ be a commutative monoid. Then the map $\tel_p(M)\to \pi_0(\tel_p(M))$ is an $\mathbb F_p$-homology isomorphism. 
\end{thm}

Combining \Cref{thm:homology} with \Cref{thm:plus}, we get:

\begin{cor}\label{cor: Fp-homol}
    For any commutative monoid $M$ we have
    \begin{gather*}
        H_*(M[\frac{1}{p}];\mathbb F_p) \simeq H_*(\tel_p(M);\mathbb F_p) = 0 \text{  when } * > 0; \\
        H_0(M[\frac{1}{p}];\mathbb F_p) \simeq H_0(\tel_p(M);\mathbb F_p) \simeq \colim (H_0(M;\mathbb F_p) \to H_0(M;\mathbb F_p) \to \dots)
    \end{gather*}
    where the colimit is taken along Frobenius maps $p: M \to M$.
\end{cor}



\begin{rmk}
    \Cref{thm:homology} can be deduced from the combination of  \Cref{thm:plus} with \Cref{cor:homologypperf}, and also  from \cite[Corollary~2.10]{ben}. However, we wanted to provide another proof because it has a different flavour: in this section, we make use of the structure of bialgebra on $H_*(M;\mathbb F_p)$, and the fact that $H_*(\tel_p(M);\mathbb F_p)$ can be interpreted in terms of this structure. 
\end{rmk}

\begin{rmk}
    We note that Tyler Lawson gave in \cite{Tyler} a proof that the monoid $Q[p_Q^{-1}]$ from \Cref{section:module} is discrete (and therefore isomorphic to $\mathbb Z$). That proof is closely related to our proof of \Cref{thm:homology}.
\end{rmk}

\begin{cons}
   There is a refinement of $\mathbb F_p$-homology to a symmetric monoidal functor $$H_*(-;\mathbb F_p): \Ss\to \coCAlg^\gr(\mathbb F_p)$$ because $\mathbb F_p$ is a field. In particular, it sends commutative monoids in $\Ss$ to commutative, cocommutative bialgebras in graded $\mathbb F_p$-vector spaces. 
\end{cons}
\begin{defn}
    Let $H$ be a commutative, cocommutative bialgebra in some symmetric monoidal category. It has a \textit{bialgebra Frobenius map} $\Phi_p$ given by $H\xrightarrow{\Delta^{(p)}}H^{\otimes p}\xrightarrow{\mu_p} H$, where $\Delta^{(p)}$, resp. $\mu_p$, is the iterated comultiplication, resp. multiplication. 
\end{defn}
\begin{obs}
    Let $M$ be a commutative monoid. The Frobenius map $p:M\to M$ induces the bialgebra Frobenius map on $H_*(M;\mathbb F_p)$.

    Therefore, we can describe $H_*(\tel_p(M);\mathbb F_p)$ entirely in terms of the bialgebra structure on $H_*(M;\mathbb F_p)$, as a colimit along bialgebra Frobenii. 
\end{obs}
\begin{nota}
    From now on, all our homology will be with coefficients in $\mathbb F_p$. 
\end{nota}

To prove the theorem, we will need a simple observation about the comultiplication of $H$. 

\begin{obs}
Let $X$ be a space. $H_*(X)$ is a graded cocommutative coalebra. Its degree $0$ part is generated by grouplike elements, that is, elements $x$ such that $\Delta(x)= x\otimes x$, namely, they are the generators corresponding to connected components of $x$.  

Furthermore, one can prove relatively easily using the Künneth formula that in the decomposition $H_n(X\times X)\cong \bigoplus_{p+q= n}H_p(X)\otimes H_q(X)$, for $x\in H_n(C) \subset \bigoplus_{D\in\pi_0(X)}H_n(D)$, we have \[\Delta(x) = x\otimes [C] + [C]\otimes x+ \text{terms in } H_p(X)\otimes H_q(X), \, p,q<n,\] where $[C]\in H_0(X)$ is the corresponding generator. 
\end{obs}

We now axiomatize this in the following definition: 
\begin{defn}
    Let $H_*$ be a graded commutative and cocommutative bialgebra. An element $x\in H_n$ is called \textit{weakly primitive} if $\Delta(x) = x\otimes \alpha + \alpha\otimes x+$ terms in $H_p\otimes H_q, p,q<n$ where $\alpha$ is a grouplike element in $H_0$, i.e., $\Delta(\alpha) = \alpha\otimes\alpha$. 
\end{defn}
\begin{ex}\label{ex:weakprim}
    We have essentially argued that the homogeneous elements of $H_*(M)$ that are ``$\pi_0(M)$-homogeneous'' are weakly primitive whenever $M$ is a commutative monoid. 
\end{ex}

Our key ingredient is the following proposition.
\begin{prop}\label{prop: weakly prim}
    Let $H_*$ be a graded commutative and cocommutative bialgebra, and let $x\in H_n, n\geq 1$ be weakly primitive. In this case, for any $m\geq 1, \Phi_m(x)= m\alpha^{m-1}x +$ terms in the ideal generated by the $H_p$'s, $0<p<n$. 
\end{prop}
\begin{proof}
    We prove a stronger statement on $\Delta^{(m)}(x)\in H^{\otimes m}$ by induction. Namely, let $t_i^m(\alpha,x)$ denote the tensor $\alpha\otimes ... \otimes x\otimes ... \otimes \alpha$ with $m$ terms and a single $x$ in position $i$. We have \[\Delta^{(m)}(x)= \sum_i t^m_i(\alpha,x) +\textnormal{ terms where at least one tensor factor is in degree } 0<p<n.\]
    Given how $\Delta^{(m)}$ is defined, we can simply prove this by induction. We note that for $i>1$, $\Delta$ applied to the first slot of $t_i^m(\alpha,x)$ is simply $t^{m+1}_{i+1}(\alpha,x)$ because $\alpha$ is grouplike, while for $i=1$, $\Delta$ applied to the first slot of $t_1^m(\alpha,x)= x\otimes \alpha...\otimes \alpha$ is $t_1^{m+1}(\alpha,x)+t_2^{m+1}(\alpha,x)+$some terms containing at least one tensor factor of degree $<n$, and in total this proves the induction step, the case $m=1$ being our assumption. 

Multiplying terms together proves the proposition, since $\alpha$ is in degree $0$ and therefore strictly commutes with $x$ so that $\mu_m(t_i^m(\alpha,x))= \alpha^{m-1}x$. 
\end{proof}
\begin{cor}\label{cor:weakprim}
     Let $H_*$ be a non-negatively graded commutative and cocommutative bialgebra over $\mathbb F_p$, and suppose that each $H_n, n\geq 1$ is generated under sums by weakly primitive elements. In this case, for every $x\in H_n$, $\Phi_p^{\circ n}(x)=\Phi_{p^n}(x) = 0$. 

     As a consequence, the colimit of $H\xrightarrow{\Phi_p}H\xrightarrow{\Phi_p}H\to \dots$ is concentrated in degree $0$. 
\end{cor}
\begin{proof}
    We prove this by induction on $n\geq 1$. For $x\in H_1$ weakly primitive, we find that the previous proposition specializes to $\Phi_p(x)= p\alpha^{p-1}x=0$, as $p=0$. 

    Now assume the statement has been proved for $m<n$, and let $x\in H_n$ be weakly primitive. It suffices to prove that $\Phi_p^{\circ n}(x) =0$, since such $x$ generate $H_n$ under sums. By \Cref{prop: weakly prim}, $\Phi_p(x)= p\alpha^{p-1}x +$ terms in the ideal generated by $H_p,0<p<n$. Since $p=0$ in $\mathbb F_p$, we may apply the induction hypothesis and the fact that $\Phi_p^{\circ n-1}$ is an algebra morphism to conclude that $\Phi_p^{\circ n}(x) = 0$, as was to be shown. 

    The claim about the colimit follows at once. 
\end{proof}
\begin{proof}[Proof of \Cref{thm:homology}]
    By \Cref{ex:weakprim}, if $M$ is a commutative monoid, $H_*(M;\mathbb F_p)$ fits into the hypotheses of \Cref{cor:weakprim}, and $H_*(\tel_p(M))$ is the colimit of $H_*(M)$ along its bialgebra Frobenius, and so \Cref{cor:weakprim} guarantees that it is concentrated in degree $0$, from which the claim follows. 

\end{proof}

\bibliographystyle{alphamod}
\let\mathbb=\mathbf
{\small
\bibliography{Biblio}
}

\end{document}